\documentclass[a4paper, reqno, 11pt]{amsart}

\usepackage[english]{babel}
\usepackage{amsmath}
\usepackage{mathrsfs}
\usepackage{amssymb}
\usepackage{amsthm}
\usepackage{enumerate}
\usepackage{ifthen}
\usepackage{bbm}
\usepackage{color}
\usepackage{xcolor}
\usepackage{graphicx}
\provideboolean{shownotes} 
\setboolean{shownotes}{true}
\usepackage{hyperref}

\newcommand{\margnote}[1]{
\ifthenelse{\boolean{shownotes}}%
{\marginpar{\raggedright\tiny\texttt{#1}}}%
{}%
}

\newcommand{\hole}[1]{
\ifthenelse{\boolean{shownotes}}%
{\begin{center} \fbox{ \rule {.25cm}{0cm}
\rule[-.1cm]{0cm}{.4cm} \parbox{.85\textwidth}{\begin{center}
\texttt{#1}\end{center}} \rule {.25cm}{0cm}}\end{center}}
{}
}
\newtheorem{thm}{Theorem}[section]

\newtheorem{prop}[thm]{Proposition}
\newtheorem{lem}[thm]{Lemma}

\newtheorem{rem}[thm]{Remark}

\theoremstyle{definition}
\newtheorem{defn}[thm]{Definition}

\newcommand{\e}{\varepsilon}		       
\newcommand{\R}{\mathbb{R}}
\newcommand{\T}{\mathbb{T}^2}
\newcommand{\N}{\mathbb{N}}
\newcommand{\Z}{\mathbb{Z}}
\newcommand{\dive}{\mathop{\mathrm {div}}}
\newcommand{\curl}{\mathop{\mathrm {curl}}}

\newcommand{\weaktos}{\stackrel{*}{\rightharpoonup}}
\newcommand{\de}{\,\mathrm{d}}
\DeclareMathOperator*{\esssup}{ess\,sup}

\numberwithin{equation}{section}


\begin{document}

\title[On weak solutions obtained via the vortex method]{Weak solutions obtained by the vortex method for the 2D Euler equations are Lagrangian and conserve the energy}

\author[G. Ciampa]{Gennaro Ciampa}
\address[G. Ciampa]{GSSI - Gran Sasso Science Institute\\ Viale Francesco Crispi 7 \\67100 L'Aquila \\Italy \& Department Mathematik Und Informatik\\ Universit\"at Basel \\Spiegelgasse 1 \\CH-4051 Basel \\ Switzerland }
\email[]{\href{gennaro.ciampa@}{gennaro.ciampa@gssi.it}}

\author[G. Crippa]{Gianluca Crippa}
\address[G. Crippa]{Department Mathematik Und Informatik\\ Universit\"at Basel \\Spiegelgasse 1 \\CH-4051 Basel \\ Switzerland}
\email[]{\href{gianluca.crippa@}{gianluca.crippa@unibas.ch}}

\author[S. Spirito]{Stefano Spirito}
\address[S. Spirito]{DISIM - Dipartimento di Ingegneria e Scienze dell'Informazione e Matematica\\ Universit\`a  degli Studi dell'Aquila \\Via Vetoio \\ 67100 L'Aquila \\ Italy}
\email[]{\href{stefano.spirito@}{stefano.spirito@univaq.it}}

\begin{abstract}
We discuss the Lagrangian property and the conservation of the kinetic energy for solutions of the 2D incompressible Euler equations. Existence of Lagrangian solutions is known when the initial vorticity is in $L^p$ with $1\leq p\leq \infty$. Moreover, if $p\geq 3/2$ \emph{all} weak solutions are conservative. In this work we prove that solutions obtained via the vortex method are Lagrangian, and that they are conservative if $p>1$.
\end{abstract}

\maketitle

\section{Introduction}
The two-dimensional Euler equations
\begin{equation}\label{eq:eu}
\begin{cases}
\partial_t v+(v\cdot\nabla)v +\nabla p =0, \\
\dive v =0,\\
v(0,\cdot) =v_0.
\end{cases}
\end{equation}
model the motion of an incompressible inviscid fluid. The unknowns are the velocity field $v:[0,T]\times\R^2\to\R^2$ and the scalar pressure $p:[0,T]\times\R^2\to\R$. In two dimensions, a very special role is played by the vorticity, which is defined as 
\begin{equation}\label{def:vort}
\omega:=\curl\,v=\partial_{x_1}v_2-\partial_{x_2}v_1.
\end{equation}
Note that the vorticity is a scalar quantity and that system \eqref{eq:eu} can be rewritten in terms of $\omega$ as
\begin{equation}\label{eq:vort}
\begin{cases}
\partial_t\omega+v\cdot\nabla\omega =0,\\
v=K*\omega,\\
\omega(0,\cdot)=\omega_0,
\end{cases}
\end{equation}
where $\omega_0=\curl v_0$ and $K$ is defined as
$$
K(x)=\frac{1}{2\pi}\frac{x^{\perp}}{|x|^2}=\frac{1}{2\pi}\frac{(-x_2,x_1)}{|x|^2}.
$$
The coupling between the velocity and the vorticity given by the formula
$$
v= K*\omega
$$
is known as Biot-Savart law and it is an alternative way to express \eqref{def:vort}.

Existence and uniqueness of classical solutions of \eqref{eq:eu} is very well-known for smooth initial data and was proved first locally in time in \cite{L} and then globally in time in \cite{W}. Smooth solutions enjoy two very natural properties: the first one is that they are \emph{Lagrangian}, namely they solve the equivalent formulation of \eqref{eq:eu} given by the following system of O.D.E. 
\begin{equation}\label{eq:lagra}
\begin{cases}
\dot{X}(t,x)=v(t,X(t,x)),\\
v(t,x)=(K*\omega)(t,x),\\
\omega(t,x)=\omega^0(X^{-1}(t,\cdot)(x)),\\
X(0,x)=x,
\end{cases}
\hspace{0.7cm}
\mbox{for }t\in [0,T] \mbox{ and }x\in\R^2.
\end{equation}
The second property is that smooth solutions conserve the \emph{kinetic energy}, namely 
\begin{equation}\label{eq:cons}
\|v(t)\|_{L^2}=\|v^0\|_{L^2} \hspace{0.5cm}\mbox{ for any }t\in [0,T].
\end{equation}
When we consider solutions in weaker classes it is not clear whether they satisfy \eqref{eq:lagra} or \eqref{eq:cons}. The goal of this paper is to prove these properties for weak solutions with $L^p$ vorticity control constructed by the {\em vortex-blob approximation}. In order to clarify how this result fits in the theory of weak solutions of the two-dimensional Euler equations we give a brief overview on the state of the art for this topic.\\
\medskip
\\
In their seminal paper \cite{DPM}, DiPerna and Majda prove the existence of measure-valued solutions of \eqref{eq:eu} under the assumption of vortex-sheet initial vorticity, that is $\omega_0\in\mathcal{M}\cap H^{-1}_{\mathrm{loc}}(\R^2)$. Precisely,  they give the definition of an approximate solution sequence of the two-dimensional incompressible Euler equations and they show that this kind of approximate solutions converge to measure-valued solutions. Moreover, they give three different examples of approximation methods that satisfy their definition:
\begin{itemize}
\item[(ES)] Approximation by exact smooth solutions of \eqref{eq:eu};
\item[(VV)] Vanishing viscosity from the two-dimensional Navier-Stokes equations;
\item[(VB)] Vortex blob approximation. 
\end{itemize} 
For initial vorticities $\omega_0\in L^1\cap L^p(\R^2)$ with $1<p\leq\infty$ they proved global existence of weak solutions of \eqref{eq:eu} obtained trough the methods (ES) and (VV), while for weak solutions constructed by (VB) the same result was obtained by Beale in \cite{Be}. Note that uniqueness of weak solutions in the class considered in \cite{DPM} is still an open problem, contrary to the case $p=\infty$ in which uniqueness has been proved by Yudovich \cite{Y}.\\
\medskip
\\
Concerning the Lagrangian property \eqref{eq:lagra}, in \cite{FMN} it has been observed that when $\omega_0\in L^p(\R^2)$, with $p\geq 2$, any weak solution of the Euler equations in vorticity form \eqref{eq:vort} is renormalized in the sense of  DiPerna and Lions \cite{DPL} and admits a representation formula in terms of the flow of the velocity as in \eqref{eq:lagra}. Moreover, when $\omega_0\in L^p(\R^2)$ with $1<p<2$, all solutions obtained as limit of (ES) are Lagrangian as a consequence of the stability theorem in \cite{DPL}. The case of weak solutions produced by (ES) with $L^1$-initial vorticity is considered in \cite{BBC2}.\\
Regarding the vanishing viscosity limit, in \cite{CS} it has been proved that solutions $\omega\in L^\infty((0,T);L^p(\R^2))$ obtained via (VV) are Lagrangian if $1< p<2$, while the case $p=1$ is considered in \cite{CNSS}. Note that the Lagrangian property is non-trivial even at the linear level for the transport equation
$$
\begin{cases}
\partial_t u+b\cdot\nabla u=0,\\
u(0,\cdot)=u_0.
\end{cases}
$$
In fact, in \cite{MS, MS2, MS3} the authors show via convex-integration techniques that there exist solutions of the linear transport equation which are not Lagrangian, if the integrability of $\nabla b$ and of $u$ are much below the threshold provided by the DiPerna-Lions' theory \cite{DPL}. In particular, for the 2D Euler equations we are in the situation described in \cite{MS3} when we assume low integrability conditions on the initial vorticity, namely $\omega_0\in L^p(\R^2)$ with $1\leq p<4/3$. \\
\medskip
\\
Regarding solutions that preserve the kinetic energy, in \cite{CFLS} the authors consider \eqref{eq:eu} on the two-dimensional flat torus $\T$ and prove that \emph{all} weak solutions $v\in C_{weak}([0,T];L^2(\T))$ satisfy the energy conservation \eqref{eq:cons} if the vorticity $\omega\in L^\infty((0,T);L^p(\T))$ with $p\geq 3/2$. The proof is based on a mollification argument and the exponent $p=3/2$ is required in order to have weak continuity of a commutator term in the energy balance. The authors also give an example of the sharpness of the exponent $p=3/2$ in their argument. Moreover, they show that if $\omega\in L^\infty((0,T);L^p(\T))$, with $1<p<3/2$, solutions constructed by (ES) and (VV) conserve the energy.\\
\medskip
\\
In this paper we consider weak solutions obtained by the vortex-blob approximation (VB). We refer to Section 3 for the precise description of the vortex-blob method, which is the prototype of several important numerical schemes and is based on the idea of approximating the vorticity with a finite number of cores which evolve according to the velocity of the fluid.\\
We introduce the following definition.
\begin{defn}\label{def:VB}
Let $v\in L^\infty((0,T);L^2_{\mathrm{loc}}(\R^2))$ and $v_0\in L^2_{\mathrm{loc}}(\R^2)$. We say that $v$ is a \emph{VB-solution} of the 2D incompressible Euler equations with initial datum $v_0$ if
\begin{itemize}
\item $v$ is a weak solution of \eqref{eq:eu}
\item there exists an approximate sequence $v^\e$ constructed with the vortex-blob method such that, as $\e\to 0$ along a subsequence,
\begin{align*}
&v^\e\weaktos v \hspace{1.8cm}\mbox{in }L^\infty((0,T);L^2_{\mathrm{loc}}(\R^2)),\\
&v^\e(0,\cdot)\to v_0 \hspace{0.9cm}\mbox{in }L^2_{\mathrm{loc}}(\R^2).
\end{align*}
\end{itemize}
\end{defn}
Our main results concern the Lagrangian property and the conservation of the kinetic energy of VB-solutions. These results are contained respectively in Theorem \ref{teo:main1} and Theorem \ref{teo:main2}.\\
In order to prove that VB-solutions are Lagrangian we will not rely on a duality argument, as done in \cite{CS}. We will prove a new, to the best of our knowledge, estimate on the $L^p$ distance between the approximate vorticity obtained by vortex-blob approximation and the solution of a linear transport equation where the advecting term is the approximate velocity field obtained by the vortex-blob approximation. Moreover, we will prove the equi-integrability of the sequence of approximate vorticity constructed via (VB) and exploit the stability theorems for Lagrangian solutions of the linear transport equation contained in \cite{BC, CDL}. In particular, the equi-integrability of the approximate vorticity will also allow us to improve the existence result of Beale in \cite{Be} to the case of initial vorticity $\omega_0\in L^1\cap H^{-1}_{\mathrm{loc}}(\R^2)$.\\
In the proof of the conservation of the energy (Theorem \ref{teo:main2}), we will use a modified version of the \emph{Serfati identity} \cite{AKFL, S} in order to prove the global convergence in $L^2$ of the approximate  velocity together with a precise blow-up estimate for the velocity. We will also prove a local balance of the energy for VB-solutions when $\omega_0\in L^p(\R^2)$ with $p\geq 6/5$.\\
It is worth to notice that, even if the vortex-blob is a numerical scheme that does not come from physical considerations, it provides solutions that are Lagrangian and conservative, two important physical properties. We think that it is an interesting problem to investigate whether in general there is any implication between Lagrangian and conservative solutions.

\subsection*{Organization of the paper}
The paper is divided as follows. In Section~2 we fix the notations and we recall some results about the linear transport equation. In Section 3 we describe the vortex-blob approximation and we prove some preliminary estimates from \cite{Be}; then we prove the equi-integrability of the approximate vorticity and the extension of Beale's result to the case of $\omega_0\in L^1\cap H^{-1}_{\mathrm{loc}}(\R^2)$. In Section 4 we prove that VB-solutions are Lagrangian and in Section 5 that they are conservative.

\bigskip

\section{Notations and Preliminaries}
This section is divided in two subsections: in the first one we fix the notations used in the sequel, while in the second one we recall the definitions of distributional, Lagrangian and renormalized solutions to the transport equation. We focus our attention on the case when the vector field is divergence-free, but all definitions and results can be extended to the case of bounded divergence with suitable changes.
\subsection{Notations}
We will denote by $L^p(\R^n)$ the standard Lebesgue spaces and with $\|\cdot\|_{L^p}$ their norm. We will use the notation $\|\cdot\|_{L^p(A)}$ when the norm is computed on a subset $A\subset\R^n$. Moreover, $L^p_c(\R^n)$ denotes the space of $L^p$ functions defined on $\R^n$ with compact support. The Sobolev space of $L^p$ functions with distributional derivatives of first order in $L^p$ is denoted by $W^{1,p}(\R^n)$. The spaces $L^p_{\mathrm{loc}}(\R^n),W^{1,p}_{\mathrm{loc}}(\R^n)$ denote the space of functions which are locally in $L^p(\R^n),W^{1,p}(\R^n)$ respectively. We will denote by $H^1(\R^n)$ the space $W^{1,2}(\R^n)$ and by $H^{-1}(\R^n)$ its dual space. Moreover, we will say that a function $u$ is in $H^{-1}_{\mathrm{loc}}(\R^n)$ if $\rho u\in H^{-1}(\R^n)$ for every function $\rho\in C^\infty_c(\R^n)$.
We also denote by $\mathcal{M}(\R^n)$ the space of finite Radon measures on $\R^n$. We denote by $L^p((0,T);L^q(\R^n))$ the space of all measurable functions $u$ defined on $[0,T]\times\R^n$ such that
$$
\|u\|_{L^p((0,T);L^q(\R^n))}:=\left(  \int_0^T\|u(t,\cdot)\|^p_{L^q} \de t\right)^{\frac{1}{p}}<\infty,
$$
for all $1\leq p<\infty$, and
$$
\|u\|_{L^\infty((0,T);L^q(\R^n))}:=\esssup_{t\in [0,T]}\|u(t,\cdot)\|_{L^q}<\infty,
$$
and analogously for the spaces $L^p((0,T);W^{1,q}(\R^n))$.
We denote by $B_R$ the ball of radius $R>0$ and center the origin in $\R^n$, by $\mathscr{L}^n$ the standard Lebesgue measure in $\R^n$, and for  $f:\R^n\to\R^n$ we consider the push-forward measure of $\mathscr{L}^n$ defined by
$$
f_\#\mathscr{L}^n(A)=\mathscr{L}^n(f^{-1}(A)),\hspace{0.7cm}\mbox{for all Borel sets }A\subseteq\R^n.
$$
Finally, it is useful to denote with $\star$ the following variant of the convolution
\begin{align*}
v\star w&=\sum_{i=1}^2 v_i*w_i \hspace{1cm}\mbox{if }v,w\mbox{ are vector fields in }\R^2,\\
A\star B&=\sum_{i,j=1}^2A_{ij}*B_{ij} \hspace{0.5cm}\mbox{if }A,B\mbox{ are matrix-valued functions in }\R^2.
\end{align*}
With the notations above it is easy to check that if $f:\R^2\to\R$ is a scalar function and $v:\R^2\to\R^2$ is a vector field, then
$$
f*\curl v=\nabla^\perp f\star v,
$$
$$
\nabla^\perp f\star \dive(v\otimes v)=\nabla\nabla^\perp f\star(v\otimes v).
$$

\medskip

\subsection{Linear Transport Equation}
Consider the Cauchy problem for the linear transport equation
\begin{equation}
\begin{cases}
\partial_t u+ b \cdot \nabla u=0, \\
u(0,\cdot)=u_0,
\end{cases}
\label{eq:te}
\end{equation}
where the vector field $b:[0,T]\times\R^n\rightarrow\R^n$ and the initial datum $u_0:\R^n\rightarrow\R$ are given. The Cauchy-Lipschitz theory gives existence and uniqueness of smooth solutions of $\eqref{eq:te}$, provided the vector field is Lipschitz in space uniformly in time. When the vector field is not Lipschitz, classical solutions do not exist in general and weaker definitions of solutions must be considered. We start with the definition of distributional solutions. 
\begin{defn}\label{def:ws}
Let $b\in L^1_{\mathrm{loc}}((0,T);L^p_{\mathrm{loc}}(\R^n))$ be a divergence-free vector field and $u_0\in L^q(\R^n)$, where $1/p+1/q \leq 1$. The function $u\in L^{\infty}((0,T);L^q(\R^n))$ is a distributional solution of \eqref{eq:te} if for any $\varphi\in C^\infty_c([0,T)\times\R^n)$ the following equality holds: 
$$
\int_0^T \int_{\R^n} u(\partial_t\varphi+b\cdot\nabla\varphi )\de x \de t+\int_{\R^n} u_0\varphi_{|_{t=0}} \de x=0.
$$
\end{defn}
The existence of global weak solutions in the sense of the previous definition is proved in \cite{DPL}. We note that the definition of distributional solution requires that the product $ub\in L^1_{\mathrm{loc}}$: this is in general not true in several applications, as in the case of the 2D Euler equations. For this reason in \cite{DPL} the authors introduce also the concept of renormalized solutions.
\begin{defn}
Let $b\in L^1_{\mathrm{loc}}((0,T);L^1_{\mathrm{loc}}(\R^n))$ be a divergence-free vector field and $u_0\in L^q(\R^n)$ for some $q\geq 1$. A function $u\in L^\infty((0,T);L^q(\R^n))$  is called renormalized solution of $(\ref{eq:te})$ if for any $\beta\in C^1(\R)\cap L^\infty(\R)$, vanishing in a neighbourhood of $0$, the equality
$$
\int_0^T \int_{\R^n} \beta(u)(\partial_t\varphi+b\cdot\nabla\varphi) \de x \de t+\int_{\R^n} \beta(u_0)\varphi_{|_{t=0}} \de x=0
$$
holds for any $\varphi\in C^\infty_c([0,T)\times\R^n)$.
\end{defn}
It is worth noticing that, when $u$ and $b$ satisfy the integrability hypothesis in Definition \ref{def:ws}, renormalized solutions are distributional solutions.\\
Finally we give the definition of Lagrangian solutions, which encodes at a weak level the fact that the solution of \eqref{eq:te} admits a representation formula in terms of the flow of the vector field $b$. We start by giving the definition of regular Lagrangian flow introduced in \cite{Am}.
\begin{defn}
Let $b\in L^1((0,T);L^1_{\mathrm{loc}}(\R^n))$ be a divergence-free vector field. We say that $X:(0,T)\times\R^n\rightarrow\R^n$ is a regular Lagrangian flow associated to $b$ if
\begin{enumerate}
\item for a.e. $x\in\R^n$ the map $t\mapsto X(t,x)$ is an absolutely continuous integral solution of the ordinary differential equation
\begin{equation}
\begin{cases}
\frac{\de}{\de t}X(t,x)=b(t,X(t,x)), \\
X(0,x)=x,
\end{cases}
\end{equation}
\item the push-forward measure $X(t,\cdot)_\#\mathscr{L}^n$
\begin{equation}\label{eq:incom}
X(t,\cdot)_\# \mathscr{L}^n=\mathscr{L}^n.
\end{equation}
\end{enumerate}
\end{defn}
Now we are ready to give the definition of Lagrangian solutions of the transport equation \eqref{eq:te}. 
\begin{defn}
Let $u_0\in L^q(\R^n)$ be given. A function $u$ is called a {\em Lagrangian solution} of (\ref{eq:te}) if $u\in L^\infty((0,T);L^q(\R^n))$ and there exists an a.e. invertible regular Lagrangian flow $X$ associated to $b$ such that
$$
u(t,x)=u_0(X^{-1}(t,\cdot)(x))
$$
for all $t\in (0,T)$ and a.e. $x\in\R^n$, where $X^{-1}(t,\cdot)$ denotes the inverse map in space at a fixed time $t$.
\end{defn}

Next, we recall a stability  result for Lagrangian solutions of \eqref{eq:te}. We start by stating the hypothesis on the vector field $b$ which will be often used in the following:
\begin{itemize}
\item[(R1)] The vector field $b$ can be decomposed as
$$
\frac{|b(t,x)|}{1+|x|}=b_1(t,x)+b_2(t,x),
$$
with $b_1\in L^1((0,T);L^1(\R^n))$ and $b_2\in L^1((0,T);L^\infty(\R^n))$.
\bigskip
\item[(R2a)] The vector field $b$ satisfies
$$
b\in L^1((0,T);W^{1,p}_{\mathrm{loc}}(\R^n)) \ \mathrm{for \ some } \ p>1.
$$
\medskip
\item[(R2b)] For every $i,j=1,...,n$ we have
$$
\partial_j b^i=S_j^i g \hspace{0.5cm} \mathrm{in} \ \mathcal{D}'((0,T)\times\R^n),
$$
where $S_j^i$ are singular integral operators of fundamental type in $\R^n$ (acting as operators in $\R^n$ at fixed time) and the function $g\in L^1((0,T);L^1(\R^n))$. See \cite{BC} for the main definitions.
\bigskip
\item[(R3)] The vector field $b$ satisfies
$$
b\in L^p_{\mathrm{loc}}((0,T)\times\R^n) \ \mathrm{for \ some} \ p>1.
$$
\end{itemize}
\medskip

The stability theorem for Lagrangian solutions of the transport equation \eqref{eq:te} that we will use in the sequel is the following, see \cite{CDL, BC} for the proof.
\begin{thm}\label{prop:stab}
Let $b_\e,b$ be divergence-free vector fields satisfying assumptions (R1), (R2a) or (R2b), (R3). Assume that $b_\e\to b$ in $L^1((0,T);L^1_{\mathrm{loc}}(\R^n))$ and that for some decomposition $\frac{|b_\e(t,x)|}{1+|x|}=b_{\e,1}(t,x)+b_{\e,2}(t,x)$ as in assumption (R1) we have that
$$
\| b_{\e,1}\|_{L^1((0,T);L^1(\R^n))}+\|b_{\e,2}\|_{L^1((0,T);L^\infty(\R^n))} \leq C.
$$
Consider a Lagrangian solution $u^\e$ of \eqref{eq:te} with coefficient $b_\e$ and initial datum $u_0^\e\in L^q(\R^n)$, as well as $u$ associated to $b$ and $u_0\in L^q(\R^n)$. If $u_0^\e\to u_0$ in $L^q(\R^n)$ with $1\leq q<\infty$, then $u^\e\to u$ in $C([0,T];L^q(\R^n))$.
\end{thm}
We conclude this subsection with a technical lemma which gives an estimate on the measure of the superlevels of a regular Lagrangian flow $X$; the proof can be found in \cite{CDL}. Define the set $G_\lambda$ as
$$
G_\lambda:=\{ x\in \R^n:|X(t,x)|\leq \lambda \mbox{ for almost every }t\in[0,T] \}.
$$
\begin{lem}\label{lem:lag}
Let $b:(0,T)\times\R^n\to\R^n$ be a vector field which admits a decomposition as in (R1) and let $X$ be a regular Lagrangian flow relative to $b$ with compression constant $L$. Then for every $r,\, \lambda>0$ it holds
$$
\mathscr{L}^n(B_r\setminus G_\lambda)\leq g(r,\lambda),
$$
where the function $g$ depends on $\| b_1\|_{L^1((0,T);L^1(\R^n))}, \| b_2\|_{L^1((0,T);L^\infty(\R^n))}$ and $L$, and satisfies $g(r,\lambda)\to 0$ for fixed $r$ and $\lambda\to\infty$.
\end{lem}

\section{The vortex blob method}
This section is devoted to the description of the vortex blob approximation and some of its properties. 
\subsection{Description of the method.} 
Consider an initial vorticity $\omega_0\in L^p_c(\R^2)$ with $1\leq p\leq \infty$. Let $\e\in (0,1)$, we consider two small parameters in $(0,1)$, which later will be chosen as functions of $\e$, denoted by $\delta(\e)$ and $h(\e)$.\\
First of all, we consider the lattice
$$
\Lambda_h:=\{\alpha_i\in\Z\times\Z:\alpha_i=h(i_1,i_2), \mbox{ where }i_1,i_2\in\Z  \},
$$
and define $R_i$ the square with sides of lenght $h$ parallel to the coordinate axis and centered at $\alpha_i\in \Lambda_h$. Let $j_{\delta}$ be a standard mollifier and define 
\begin{equation}\label{eq:idv}
\omega_0^\e:=\omega_0*j_{\delta(\e)}.
\end{equation}
For any $\delta\in(0,1)$ the support of $\omega_0^\e$ is contained in a fixed compact set in $\R^2$, then it can be tiled by a finite number $N(\e)$ of squares $R_i$. Define the quantities
$$
\Gamma^\e_i=\int_{R_i} \omega_0^\e(x) \ \de x, \hspace{0.5cm}\mbox{for }i=1,..., N(\e).
$$
Let $\varphi_\e$ be another mollifier, we define the approximate vorticity to be
\begin{equation}\label{def:vb}
\omega^\e(t,x)=\sum_{i=1}^{N(\e)} \Gamma_i^\e \varphi_\e(x-X^\e_i(t)), 
\end{equation}
where $\{X^\e_i(t)\}_{i=1}^{N(\e)}$ is a solution of the O.D.E. system
\begin{equation}\label{eq:vb}
\begin{cases}
\dot{X}^\e_i(t)=v^\e(t,X^\e_i(t)), \\
X^\e_i(0)=\alpha_i,
\end{cases}
\end{equation}
with $v^{\e}$ defined as 
\begin{equation}\label{eq:av}
v^\e(t,x)=K*\omega^\e(t,x)=\sum_{i=1}^{N(\e)} \Gamma_i^\e K_\e(x-X^\e_i(t)),
\end{equation}
where $K_\e=K*\varphi_\e$.
Note that, since $\delta$ and $h$ are $\e$-dependent, we only use the superscript $\e$. The ordinary differential equations \eqref{eq:vb} are known as the {\em vortex-blob approximation}.\\
It is not difficult to show the bound
\begin{equation}\label{eq:epsvreg}
\sup_{t\in[0,T]}\left( \|v^\e(t,\cdot)\|_{L^\infty}+\|\nabla\,v^\e(t,\cdot)\|_{L^\infty}\right)\leq\frac{C}{\e^2}, 
\end{equation}
see \cite{DPM}. From \eqref{eq:epsvreg} it follows that, for every fixed $\e>0$, there exists a unique smooth solution $\{X^\e_i(t)\}_{i=1}^{N(\e)}$ of the O.D.E. system \eqref{eq:vb}, which implies that $v^{\e}$ and $\omega^{\e}$ are well-defined smooth functions. Note that $v^\e$ and $\omega^\e$ are not exact solutions of the Euler equations because of the presence of an error term, due to the fact that each blob is rigidly translated by the flow. Precisely, the approximate vorticity $\omega^\e$ satisfies the following equation
\begin{equation}\label{eq:wvb}
\partial_t\omega^\e+v^\e\cdot\nabla\omega^\e=E_\e,
\end{equation}
where by a direct computation the error term is given by
\begin{equation}\label{eq:ee}
E_\e(t,x):=\sum_{i=1}^{N(\e)}\left[v^\e(t,x)-v^\e(t,X_i^\e(t)\right]\cdot \nabla\varphi_\e(x-X^\e_i(t))\Gamma^\e_i.
\end{equation}
Concerning the approximate velocity $v^\e$, consider the quantity
$$
w^\e=\partial_t v^\e+\left(v^\e\cdot\nabla\right) v^\e.
$$
Since $w^\e$ satisfies the system
\begin{equation}
\begin{cases}
\curl w^\e=E_\e,\\
\dive w^\e=\dive \dive \left(v^\e\otimes v^\e \right),
\end{cases}
\end{equation}
we derive that there exists a function $p^\e$ such that
$$
-\Delta p^\e=\dive \dive \left(v^\e\otimes v^\e \right),
$$
and
$$
w^\e=-\nabla p^\e+K*E_\e.
$$
Then, the velocity given by the vortex-blob approximation verifies the following equations
\begin{equation}\label{eq:vbv}
\begin{cases}
\partial_t v^\e+\left(v^\e\cdot\nabla\right) v^\e+\nabla p^\e=K*E_\e,\\
\dive v^\e=0.
\end{cases}
\end{equation}
Since $v^\e$ is divergence-free, $E_\e$ can be rewritten as $E_\e(t,x)=\dive F_\e(t,x)$ where
\begin{equation}\label{eq:fe}
F_\e(t,x):=\sum_{i=1}^{N(\e)}\left[v^\e(t,x)-v^\e(t,X_i^\e(t)\right]\varphi_\e(x-X^\e_i(t))\Gamma^\e_i.
\end{equation}
\subsection{A priori estimates} 
In this subsection we give the proof of some a priori estimates on $\omega^\e$, $v^\e$, and the error term $F_\e$, taken from \cite{Be}. First of all, we introduce the following auxiliary problem. Let  $\bar{\omega}^\e$ be the solution of the linear transport equation with vector field $v^{\e}$, that is 
\begin{equation}
\begin{cases}
\partial_t \bar{\omega}^\e+ v^\e \cdot \nabla \bar{\omega}^\e=0, \\
\bar{\omega}^\e(0,\cdot)=\omega_0^\e.
\end{cases}
\label{eq:1}
\end{equation}
Since $v^{\e}$ satisfies \eqref{eq:epsvreg}, there exists a unique smooth solution $\bar{\omega}^\e$, which is given by the formula
\begin{equation}\label{eq:baromega}
\bar{\omega}^\e(t,x)= \omega_0^\e((X^\e)^{-1}(t,\cdot)(x)),
\end{equation}
where $X^\e$ is the flow of $v^\e$, that is,
\begin{equation}\label{eq:fv}
\begin{cases}
\dot{X}^\e(t,x)=v^\e(t,X^\e(t,x)), \\
X^\e(0,x)=x.
\end{cases}
\end{equation}
Moreover, since $\dive v^\e=0$, we have
$$
\|\bar{\omega}^\e(t,\cdot)\|_{L^p}=\|\omega_0^\e\|_{L^p}\leq \|\omega_0\|_{L^p}.
$$
We will use $\bar{\omega}^\e$ in order to prove uniform $L^{p}$-bounds on $\omega^\e$. Before doing that, note that $\omega^\e$ can be seen as a discretization of $\varphi_\e*\bar{\omega}^\e$, since a change of variables gives
\begin{equation}\label{eq:int}
\varphi_\e*\bar{\omega}^\e(t,x)=\int_{\R^2}\varphi_\e(x-z)\bar{\omega}^\e(t,z)\de z=\int_{\R^2}\varphi_\e(x-X^\e(t,y))\omega_0^\e(y)\de y,
\end{equation}
compare with \eqref{def:vb}. We now give a lemma which is, loosely speaking, an estimate on the $L^p$ norms of the error we commit substituting the integral in \eqref{eq:int} with the sum in \eqref{def:vb}. The following estimate is new for $1\leq p<\infty$, while the case $p=\infty$ has been proved in \cite{Be}.
\begin{lem}\label{lem:est}
Let $\omega_0\in L^{1}(\R^{2})$ and let $h=h(\e)$ be chosen as
\begin{equation}\label{eq:h}
h(\e)=\frac{\e^4}{\exp\left(C_1\e^{-2}\|\omega_0\|_{L^1}T\right)},
\end{equation}
where $C_1>0$ is a positive constant. Then, the estimate
\begin{equation}\label{es:1}
\sup_{0\leq t\leq T} \| \omega^\e-\varphi_\e*\bar{\omega}^\e\|_{L^p} \leq C \e^{1+\frac{2}{p}}
\end{equation}
holds for all $1\leq p\leq \infty$, where $C>0$ is a positive constant which does not depend on $\e$.
\end{lem}
\begin{proof}
We start by proving the inequality \eqref{es:1} in the case $p=1$. By using the definitions of $\omega^\e$ and $\bar{\omega}^\e$ we have that
\begin{align}
& \int_{\R^2} \left|\omega^\e(t,x)-\varphi_\e*\bar{\omega}^\e(t,x)\right|\de x  \nonumber \\
& = \int_{\R^2}\left|
\displaystyle\sum_i \Gamma^\e_i \varphi_\e(x-X^\e_i(t))-\int_{\R^2} \varphi_\e(x-z)\bar{\omega}^\e(t,z) dz
\right| \de x \nonumber \\
& = \int_{\R^2} \left|\displaystyle\sum_i \int_{R_i}\omega_0^\e(y)\de y \varphi_\e(x-X^\e_i(t))-\int_{\R^2}\varphi_\e(x-z)\omega_0^\e((X^\e)^{-1}(t,\cdot)(z)) dz\right| \de x \nonumber \\
&=\int_{\R^2} \left|
\displaystyle\sum_i \int_{R_i} \omega_0^\e(y)\varphi_\e(x-X^\e_i(t))\de y -\int_{\R^2} \varphi_\e(x-X^\e(t,y))\omega_0^\e(y) \de y \right|\de x \nonumber \\
&=\int_{\R^2} \left|
\displaystyle\sum_i \int_{R_i} \omega_0^\e(y)\left[\varphi_\e(x-X^\e_i(t))-\varphi_\e(x-X^\e(t,y))\right] \de y \right| \de x \nonumber \\
&\leq \int_{\R^2} \sum_i \int_{R_i} \left| \omega_0^\e(y)\right| |\underbrace{\varphi_\e(x-X^\e_i(t))-\varphi_\e(x-X^\e(t,y))}_{(*)}| \de y \de x \label{resc}.
\end{align}
For $(\ast)$ we have the following estimate
\begin{align*}
& \varphi_\e(x-X^\e_i(t))-\varphi_\e(x-X^\e(t,y)) \\
&= \int_0^1 \nabla\varphi_\e(x-X^\e(t,y)+s(X^\e(t,y)-X^\e_i(t))) \de s \left(X^\e(t,y)-X^\e_i(t)\right) \\
& =\e^{-3} \int_0^1 \nabla\varphi\left(\frac{s(x-X^\e_i(t))+(1-s)(x-X^\e(t,y))}{\e}\right) \de s \left(X^\e(t,y)-X^\e_i(t)\right).
\end{align*}
So, for any $y\in R_i$ we have that
$$
\begin{vmatrix}
 X^\e(t,y)-X^\e_i(t)
\end{vmatrix}\leq C \,\mathrm{Lip}(X^\e(t,\cdot))\, h,
$$
where $\mathrm{Lip}(X^\e(t,\cdot))$ is the Lipschitz constant of the flow $X^\e(t,\cdot)$, which is bounded by
\begin{equation}
\mathrm{Lip}(X^\e(t,\cdot))\leq \exp\left(C\e^{-2}\|\omega_0\|_{L^1}T\right),
\end{equation}
as a consequence of \eqref{eq:epsvreg}. Then, rescaling in the $x$ variable in \eqref{resc} we have
$$
\|\omega^\e(t,\cdot)-\varphi_\e*\bar{\omega}^\e(t,\cdot)\|_{L^1} \leq h \,\e^{-1}\,\mathrm{Lip}(X^\e(t,\cdot))\, \|\nabla\varphi\|_{L^1}\|\omega_0\|_{L^1}.
$$
Choosing the function $h$ as in \eqref{eq:h} we get \eqref{es:1} for $p=1$.\\
For $p=\infty$ we can argue in a similar way; by the same computations as in \eqref{resc} we have that
\begin{align*}
&|\omega^\e(t,x)-\varphi_\e*\bar{\omega}^\e(t,x)|\\ &\leq\displaystyle\sum_{i=1}^{N(\e)}\int_{R_i}|\omega_0^\e(y)||\underbrace{\varphi_\e(x-X^\e_i(t))-\varphi_\e(x-X^\e(t,y))}_{(*)}|\de y,
\end{align*}
and we can estimate $(\ast)$ as before, so that
$$
\|\omega^\e(t,\cdot)-\varphi_\e*\bar{\omega}^\e(t,\cdot)\|_{L^\infty} \leq h \,\e^{-3}\,\mathrm{Lip}(X^\e(t,\cdot)) \,\|\nabla\varphi\|_{L^\infty}\|\omega_0\|_{L^1},
$$
and choosing $h$ as in \eqref{eq:h} we get the result.\\
Finally, by interpolating we have
$$
\|\omega^\e-\varphi_\e*\bar{\omega}^\e\|_{L^p}\leq \|\omega^\e-\varphi_\e*\bar{\omega}^\e\|_{L^1}^{\frac{1}{p}}\|\omega^\e-\varphi_\e*\bar{\omega}^\e\|_{L^\infty }^{1-\frac{1}{p}}\leq C\e^{1+\frac{2}{p}},
$$
and this concludes the proof.
\end{proof}
We are now in the position to prove the uniform $L^p$ bound on $\omega^\e$.
\begin{lem}\label{lem:bound}
Let $\omega_0 \in L^p_c(\R^2)$. Then, the approximate vorticities $\omega^\e$ defined in \eqref{def:vb} satisfy the following
$$
\sup_{0\leq t\leq T}\| \omega^\e(t,\cdot)\|_{L^p}\leq C\left( \|\omega_0\|_{L^p}+\|\omega_0\|_{L^1}^{\frac{1}{p}}\right)
$$
for $1\leq p<\infty$, and
$$
\sup_{0\leq t\leq T}\|\omega^\e(t,\cdot)\|_{L^\infty}\leq C\|\omega_0\|_{L^\infty}.
$$
\end{lem}
\begin{proof}
First of all, the case $p=1$ follows directly from the definition of $\omega^\e$ since
$$
\|\omega^\e(t,\cdot)\|_{L^1}\leq \sum_i^{N(\e)} | \Gamma^\e_i | \int \varphi_\e(x-X^\e_i(t)) \de x \leq \|\omega_0\|_{L^1}.
$$
Let consider now $1<p<\infty$ and let $A(t)$ and $B(t)$ be the sets
$$
A(t):=\{ x\in \R^2:|\omega^\e(t,x)|>1 \},
$$
$$
B(t):=\{ x\in \R^2:|\omega^\e(t,x)|\leq 1 \}.
$$
By Chebishev inequality
$$
\mathscr{L}^2(A(t))\leq C\|\omega_0\|_{L^1},
$$
uniformly in time. Let $\bar{\omega}^\e$ the solution of $(\ref{eq:1})$, we have that
\begin{align*}
\|\omega^\e(t,\cdot)\|_{L^p(A(t))} & \leq \|\varphi_\e*\bar{\omega}^\e(t,\cdot)\|_{L^p(A(t))}+ \|\omega^\e(t,\cdot)-\varphi_\e*\bar{\omega}^\e(t,\cdot)\|_{L^p(A(t))} \\
& \leq C\|\omega_0\|_{L^p}+\mathscr{L}^2(A(t))^{\frac{1}{p}}\|\omega^\e-\varphi_\e*\bar{\omega}^\e\|_{L^\infty} \\
& \leq C \left(\|\omega_0\|_{L^p}+\e\|\omega_0\|_{L^1}^{\frac{1}{p}}\right).
\end{align*}
On the other hand, for the set $B(t)$ we have
$$
\int_B|\omega^\e(t,x)|^p\de x\leq \int_{\R^2}|\omega^\e(t,x)|\de x \leq C \|\omega_0\|_{L^1},
$$
since $|\omega^\e(t,x)|^p\leq |\omega^\e(t,x)|$ on $B(t)$. Combining the previous estimates, since $\e<1$, taking the supremum in time we have the result.\\
Finally, the case $p=\infty$ follows from the triangle inequality and \eqref{es:1}.
\end{proof}
We give now a convergence result for the error term $F_\e$ (see \cite{Be} for the proof).
\begin{lem}\label{lem:fe}
Let $\omega_0\in L^p_c(\R^2)$ with $p\geq 1$, then the quantity $F_\e$ defined in \eqref{eq:fe} satisfies
\begin{equation}\label{convFL1}
\sup_{t\in[0,T]}\|F_\e(t,\cdot)\|_{L^1}\to 0, \hspace{0.5cm}\mbox{as }\e\to 0.
\end{equation}
In particular, for $1<p<2$ we have the following bound
\begin{equation}\label{es:FL1}
\|F_\e(t,\cdot)\|_{L^1}\leq C\delta^{-\beta}\e^{\frac{1}{3}}\|\omega_0\|_{L^1},
\end{equation}
where $\beta=2\left(\frac{1}{p}-\frac{1}{3}\right)$ and $\delta(\e)=\e^\sigma$ with $0<\sigma<1/4$. Moreover, choosing $h(\e)=C_1\e^6\exp\left(-C_0\e^{-2}\right)$ where $C_1,C_0$ are positive constants, we have that $F_\e$ satisfies the following additional bound
$$
\|F_\e(t,\cdot)\|_{L^2}\leq C\delta^{-\beta}\e^{\frac{7}{3}}\|\omega_0\|_{L^1},
$$
which goes to $0$ choosing $\delta$ as above and $0<\sigma<1/7$.
\end{lem}
It is worth to note that the dependence on $p$ of the bound in \eqref{es:FL1} is due to the fact that, in order to obtain the convergence in \eqref{convFL1}, for $1\leq p<2$ we need to regularize the initial vorticity, while is not needed for $p>2$.\\
The uniform bound of Lemma \ref{lem:bound} together with Lemma \ref{lem:fe} are the core of the proof of the theorem proved by Beale in \cite{Be}, where he showed the existence of VB-solutions when the initial vorticity is in $L^p$, with $p>1$, and compactly supported. In detail:
\begin{thm}\label{teo:beale1}
Let $v_0\in L^2_{\mathrm{loc}}(\R^2)$ and assume that the vorticity $\omega_0=\curl v_0\in L^p_c(\R^2)$ for some $p>1$. Let $\omega^\e$ given by the vortex-blob approximation with parameters chosen so that $\delta(\e)=\e^\sigma$ for some $0<\sigma<1/4$, and $h(\e)\leq C\e^4\exp(-C_0\e^{-2})$ for some constants $C_0,C$. Then up to subsequences, $v^\e$ converges strongly in $L^2((0,T);L^2_{\mathrm{loc}}(\R^2))$ to a classical weak solution of the Euler equations with initial velocity $v_0$.
\end{thm}

\medskip

\subsection{The $L^1$ case}
In this subsection we consider the case of initial vorticities $\omega_0\in L^1_c(\R^2)$. In particular, we prove the equi-integrability of the sequence of approximate vorticities $\{\omega^\e\}$ given by the vortex-blob method and this will be crucial in the extension of Beale's result to the case $p=1$. Moreover, the fact that $\omega^\e$ is equi-integrable will also be fundamental for the applications of the linear theory discussed in Section 2 to the 2D Euler equations. We start by showing the equi-integrability of $\omega^\e$ in the following (up to our knowledge original) lemma.
\begin{lem}\label{lem:equi}
Let $\omega_0\in L^1_c(\R^{n})$ and $\omega_0^{\e}$ defined as \eqref{eq:idv}. Then the approximate vorticities $\omega^\e$ as in \eqref{def:vb} are equi-integrable in $L^{1}((0,T)\times\R^2)$. 
\end{lem}
\begin{proof}
We divide the proof in several steps.\\
\\
\underline{Step 1} \hspace{0.5cm} The sequence $\{ \bar{\omega}^\e \}_\e$ is equi-integrable.\\
\\
We start by proving the equi-integrability of the sequence $\bar{\omega}^\e$ on small sets; we have that
$$
\int_A|\bar{\omega}^\e(t,x)|\de x=\int_A|\omega_0^\e((X^\e)^{-1}(t,\cdot)(x))|\de x=\int_{X^\e(t,A)}|\omega_0^\e(y)|\de y.
$$
Since $v^\e$ is divergence-free we have that $\mathscr{L}^2(X^\e(t,A))=\mathscr{L}^2(A)$, so the measure of the set $X^\e(t,A)$ is indipendent from $t$ and $\e$ and then the equi-integrability of $\omega_0^\e$ gives the result.\\
We move now to the proof of the equi-integrability at infinity; we have that
\begin{align*}
\int_{\R^2\setminus B_r} & |\bar{\omega}^\e(t,x)|\de x=\int_{\R^2\setminus B_r}|\omega_0^\e((X^\e)^{-1}(t,x))|\de x = \\
& =\int_{ \{ y\in B_R:|X^\e(t,y)|>r \} } \ | \omega_0^\e(y)|\de y,
\end{align*}
where $\mathrm{supp}\ \omega_0^\e\subseteq B_R$. By Lemma \ref{lem:lag}, the measure of the set 
$$
\{ y\in B_R:|X^\e(t,y)|>r \}
$$ 
can be made arbitrary small for $r$ big enough, indipendently from $\e$ and $t$. Then by the equi-integrability of $\omega_0^\e$ the claim of the first step follows.\\
\\
\underline{Step 2} \hspace{0.5cm} The sequence $\{ \varphi_\e*\bar{\omega}^\e \}_\e$ is equi-integrable.\\
\\
We start by proving the equi-integrability of $\varphi_\e*\bar{\omega}^\e$ on small sets. Since the initial datum $\omega_0^\e$ has compact support (uniformly in $\e$) and converges strongly, therefore weakly, to $\omega_0$ in $L^1$, De la Vall\'e-Poussin's theorem provides the existence of a function $G$ positive, increasing and superlinear such that
$$
\sup_\e \int_{\R^2} G(|\omega_0^\e(x)|)\de x<\infty.
$$
Then, for $\varphi_\e*\bar{\omega}^\e$ we have that
\begin{align}
\int_{\R^2} G(| \varphi_\e*\bar{\omega}^\e|(t,x))\de x &= 
\int_{\R^2} G\left( \mid \int\varphi_\e(x-y)\bar{\omega}^\e(t,y) \de y \mid\right)\de x  \label{proof:equi1}\\
& \leq \int_{\R^2} G\left(\int\varphi_\e(x-y)| \bar{\omega}^\e(t,y)| \de y\right) \de x \label{proof:equi2}\\
& \leq \int\int\varphi_\e(x-y)G(| \bar{\omega}^\e(t,y)|) \de y\de x  \label{proof:equi3}\\
& = \int_{\R^2} G(| \bar{\omega}^\e(t,y)|)\int\varphi_\e(x-y)\de x\de y \nonumber\\
& = \int_{\R^2} G(| \bar{\omega}^\e(t,y)|)\de y. \label{proof:equi4}
\end{align}
Note that in \eqref{proof:equi2} we have taken the modulus inside the integral and used that $G$ is increasing, while in \eqref{proof:equi3} we used Jensen's inequality since $G$ is convex and $\varphi_\e(x-\cdot) \de y$ is a probability measure. Multiplying \eqref{eq:baromega} by $G'(|\bar{\omega}^\e|)$, integrating in space and using the divergence-free condition of $v^\e$, from the equi-integrability of $\omega_0^\e$ it follows that
\begin{equation}\label{proof:equi5}
\sup_\e \sup_{t\in[0,T]}\int_{\R^2} G(|\bar{\omega}^\e(t,x)|)\de x\leq\sup_\e \int_{\R^2} G(|\omega_0^\e(x)|)\de x<\infty.
\end{equation}
Then, taking the supremum in time and in $\e$ in \eqref{proof:equi1} and estimating \eqref{proof:equi4} with \eqref{proof:equi5} shows the equi-integrability on small sets. The equi-integrability at infinity is an immediate consequence of that of $\bar{\omega}^\e$.\\
\\
\underline{Step 3} \hspace{0.5cm} The sequence $\{ \omega^\e\}_\e$ is equi-integrable.\\
\\
We start by proving equi-integrability on small sets. We can compute
\begin{align*}
\int_A |\omega^\e(t,x)|\de x & \leq \int_A |\omega^\e(t,x)-\varphi_\e*\bar{\omega}^\e(t,x)|\de x+\int_A |\varphi_\e*\bar{\omega}^\e(t,x)| \de x \\ & \leq \|\omega^\e-\varphi_\e*\bar{\omega}^\e\|_{L^\infty}\mathscr{L}^2(A)+\int_A |\varphi_\e*\bar{\omega}^\e(t,x)|\de x.
\end{align*}
Fix $\eta>0$. The first term can be estimate using $\|\omega^\e-\varphi_\e*\bar{\omega}^\e\|_\infty\leq C \e\leq C$ and choosing $\gamma_1<\displaystyle\frac{\eta}{2C}$ so that for $\mathscr{L}^2(A)< \gamma_1$
$$
\|\omega^\e-\varphi_\e*\bar{\omega}^\e\|_\infty\mathscr{L}^2(A)\leq C\gamma_1\leq\frac{\eta}{2}.
$$
For the second term we use the equi-integrability of $\varphi_\e*\bar{\omega}^\e$. There exists $\gamma_2$ such that
$$
\int_A |\varphi_\e*\bar{\omega}^\e|\de x<\frac{\eta}{2},
$$
if $\mathscr{L}^2(A)\leq\gamma_2$.
So taking $\gamma=\min(\gamma_1,\gamma_2)$, assuming $\mathscr{L}^2(A)\leq\gamma$ and then taking the supremum in time, the equi-integrability on small sets is proven.\\
We prove now the equi-integrability at infinity. Fix $\eta>0$ and decompose
$$
\int_{B_R^c} |\omega^\e(t,x)|\de x\leq \int_{B_R^c} |\omega^\e(t,x)-\varphi_\e*\bar{\omega}^\e(t,x)|\de x+\int_{B_R^c} |\varphi*\bar{\omega}^\e(t,x)|\de x.
$$
Since $\varphi*\bar{\omega}^\e$ is equi-integrable there exists $R_1>0$ such that for every $R>R_1$
$$
\int_{B_R^c} |\varphi_\e*\bar{\omega}^\e(t,x)|\de x\leq \frac{\eta}{2},
$$
and by \eqref{es:1}, if we consider $\e\leq\bar{\e}:=\sqrt[3]{\frac{\eta}{2C}}$ we obtain
$$
\int_{B_R^c} |\omega^\e(t,x)|\de x\leq C\e^3+\frac{\eta}{2}\leq \eta.
$$
For $\e>\bar{\e}$ we do not use estimate \eqref{es:1} but we focus our attention on the flows $X_i^\e$. From the definition of $\omega^\e$ we know
\begin{equation}
\int_{B_R^c} |\omega^\e(t,x)|\de x\leq \sum_i |\Gamma^\e_i|\int_{B^c_R} \varphi_\e(x-X^\e_i(t))\de x.
\label{eq:equiinfty}
\end{equation}
For the flows $X^\e_i(t)$ we have that for a given finite time $T$
\begin{equation}\label{proof:Linftyboundequi}
|X^\e_i(t)|\leq |\alpha_i|+\int_0^T |v^\e(\tau,X^\e_i(\tau))|\de\tau.
\end{equation}
Since $\alpha_i\in \mathrm{supp} \ \omega_0^\e$ which is compact, we have $|\alpha_i|\leq \tilde{R}$. Decompose the Biot-Savart Kernel as $K=K\chi_{B_1}+K\chi_{B_1^c}:=K_1+K_2$ where $K_1\in L^1$ and $K_2\in L^\infty$. Using Young inequality for convolutions we get
$$
\int_0^T |v^\e(\tau,X^\e_i(\tau))|\de\tau\leq T\left( \|K_1\|_{L^1}\|\omega^\e\|_{L^\infty}+\|K_2\|_{L^\infty} \|\omega_0\|_{L^1} \right)
$$
and
$$
|\omega^\e(t,x)|\leq \frac{1}{\e^2}\sum_i |\Gamma^\e_i|\leq \frac{1}{\bar{\e}^2}\|\omega_0\|_{L^1}.
$$
Then by \eqref{proof:Linftyboundequi} we have that
$$
|X^\e_i(t)|\leq \tilde{R}+T\left( \|K_1\|_{L^1}\frac{1}{\bar{\e}^2}+\|K_2\|_{L^\infty}\right) \|\omega_0\|_{L^1}.
$$
Defining $R_2>0$ as
$$
R_2=\tilde{R}+T\left( \|K_1\|_{L^1}\frac{1}{\bar{\e}^2}+\|K_2\|_\infty\right) \|\omega_0\|_{L^1}+2,
$$ 
we have that for $|x|>R_2$
$$
|x-X^\e_i(t)|\geq R_2-\tilde{R}-T\left( \|K_1\|_{L^1}\frac{1}{\bar{\e}^2}+\|K_2\|_{L^\infty}\right) \|\omega_0\|_{L^1}>1
$$
so that in \eqref{eq:equiinfty} we integrate out of the support of $\varphi_\e$ and the integral therefore vanishes.
Setting $R=\max(R_1,R_2)$ and taking the supremum in $\e$ and $t$ we have the result since $\bar{\e}$ depends only on $\eta$.
\end{proof}
If we assume in addition that the initial vorticity $\omega_0\in H^{-1}_{\mathrm{loc}}(\R^2)$, then the initial velocity $v_0$ is locally square integrable; this is the content of the following proposition.
\begin{prop}\label{prop:vL2loc}
Let $\omega_0\in L^1_c\cap H^{-1}_{\mathrm{loc}}(\R^2)$ and $v^\e$ defined as in \eqref{eq:av}. Then $v^\e\in L^\infty((0,T);L^2_{\mathrm{loc}}(\R^2))$ and
$$
\sup_{t\in[0,T]}\|v^\e(t,\cdot)\|_{L^2(B_R)}\leq C(R).
$$
\end{prop}
\begin{proof}
First of all, we decompose $v^\e(t,x)=\tilde{v}^\e(t,x)+\bar{v}(x)$ where $\bar{v}$ is a smooth steady solution of the 2D Euler equations and $\tilde{v}^\e(t,x)$ is at each time in $L^2$ with zero total circulation. To do this, consider the same mollifier $\varphi$ as in the definition of $\omega^\e$ in \eqref{def:vb} and set
$$
\Gamma=-\int_{\R^2}\omega_0(x)\de x,\hspace{0.5cm}\bar{\omega}(x)=\Gamma\varphi(x),
$$
$$
\bar{v}(x)=K*\bar{\omega},\hspace{0.5cm}\tilde{v}^\e=v^\e-\bar{v},\hspace{0.5cm}\tilde{\omega}^\e=\omega^\e-\bar{\omega}.
$$
Then, $\tilde{v}^\e$ solves the following equation
\begin{equation}\label{eq:vtilde}
\partial_t\tilde{v}^\e+\left(v^\e\cdot\nabla\right)\tilde{v}^\e+\left(\tilde{v}^\e\cdot\nabla\right)\bar{v}+\nabla p^\e=K*\left(\dive F_\e\right).
\end{equation}
Multiplying \eqref{eq:vtilde} by $\tilde{v}^\e$ and integrating over $\R^2$ we have
\begin{equation}
\frac{1}{2}\frac{\de}{\de t}\|\tilde{v}^\e(t,\cdot)\|_{L^2}^2\leq \|\nabla\bar{v}\|_{L^\infty}\|\tilde{v}^\e(t,\cdot)\|_{L^2}^2+\|K*\left(\dive F_\e\right)\|_{L^2}\|\tilde{v}^\e(t,\cdot)\|_{L^2}.
\end{equation}
Since $F\mapsto(K*\dive F)$ is a bounded operator in $L^2$ we get
$$
\|\tilde{v}^\e(t,\cdot)\|_{L^2}\leq C(T)\|\tilde{v}^\e_0\|_{L^2}, \hspace{1cm}\mbox{for all }0\leq t\leq T. 
$$
In order to conclude, it is enough to prove that $\|\tilde{v}^\e_0\|_{L^2}$ is finite. Note that
$$
\tilde{v}^\e_0(x)=\sum_{i=1}^{N(\e)}\Gamma^\e_iK_\e(x-\alpha_i)-\Gamma K*\varphi(x).
$$
The previous sum is a discretization of the integral
$$
\int_{\R^2}K_\e(x-\alpha)\omega_0^\e(\alpha)\de \alpha=(K*\varphi_\e)*(j_\delta*\omega_0),
$$
which is by hypothesis bounded in $L^2$. Since in the definition of $\tilde{v}_0$ the kernel $\varphi$ is chosen to be the same as in \eqref{eq:vb}, the discretization error can be pointwise bounded by $h\|\omega_0\|_{L^1}\|\nabla K_\e\|_{L^\infty}=Ch\e^{-2}$, which is small by our choice of $h(\e)$. It follows that $v_0^\e$, and therefore $\tilde{v}^\e_0$ is uniformly bounded in $L^2(B_{2R})$ where $R>0$ is such that $\mathrm{supp}\,\omega_0^\e\subseteq B_R$. For $|x|>2R$, $K_\e$ is just $K$ and then
$$
\tilde{v}^\e_0(x)=\sum_{i=1}^{N(\e)}\Gamma^\e_i\left( K(x-\alpha_i)-K(x)\right),
$$
and it is easy to see that it is bounded by
$$
\sum_{i=1}^{N(\e)}C|x|^{-2}|\Gamma^\e_i|\leq \|\omega_0\|_{L^1}|x|^{-2},
$$
thus it is bounded in $L^2(B^c_{2R})$.
\end{proof}

The equi-integrability of the vortex-blob vorticity $\omega^\e$ guarantees the phenomenon of concentration-cancellations, see \cite{VW}. This fact together with the consistency of the method implies the existence of VB-solutions in the case of $L^1_c$ initial vorticity. In particular with Lemma \ref{lem:equi} we improve the result of \cite{Be} to the case $\omega_0\in L^1_c\cap H^{-1}_{\mathrm{loc}}(\R^2)$ and this is the content of the following theorem.
\begin{thm}
Let $v_0\in L^2_{\mathrm{loc}}(\R^2)$ and assume that the vorticity $\curl v_0=\omega_0\in L^1_c(\R^2)\cap H^{-1}_{\mathrm{loc}}(\R^2)$. Let $\omega^\e$ be given by the vortex-blob approximation with the parameters chosen so that $\delta(\e)=\e^\sigma$ for some $0<\sigma<1/7$, and $h(\e)\leq C \e^6 \exp(-C_0\e^{-2})$ for certain $C_0, C$. Then there exists a subsequence of $v^\e$ which converges strongly in $L^q((0,T);L^q_{\mathrm{loc}}(\R^2))$ for any $1\leq q<2$ and weakly in $L^\infty((0,T);L^2_{\mathrm{loc}}(\R^2))$ to a classical weak solution $v$ of the Euler equations with initial velocity $v_0$.
\end{thm}
\begin{proof}
We just sketch the proof since it follows the proof of \cite{Be}, \cite{De} and \cite{VW}.\\
\\
\underline{Step 1} \hspace{1cm}\emph{Compactness}.\\
\\
Since $\omega_0\in L^1_c\cap H^{-1}_{\mathrm{loc}}(\R^2)\subset\mathcal{M}\cap H^{-1}_{\mathrm{loc}}(\R^2)$, by Theorem 2 in \cite{Be} we have the existence of $v\in L^\infty((0,T);L^2_{\mathrm{loc}}(\R^2))$ such that
$$
v^\e\weaktos v \hspace{0.5cm}\mbox{in }L^\infty((0,T);L^2_{\mathrm{loc}}(\R^2)),
$$
and for every $1\leq q<2$ we have the strong convergence
$$
v^\e\to v \hspace{0.5cm} \mbox{in } L^q((0,T);L^q_{\mathrm{loc}}(\R^2)).
$$
Moreover, for every test function $\Phi\in C^\infty_c((0,T)\times\R^2)$ with $\dive\Phi=0$ using Lemma \ref{lem:fe} we have that
$$
\lim_{\e\to 0}\int_0^T\int_{\R^2}\left(\partial_t v^\e\Phi+v^\e\otimes v^\e:\nabla\Phi  \right)\de x \de t =0.
$$
\\
\underline{Step 2}\hspace{1cm}\emph{Convergence}.\\
\\
In order to conclude we have to prove the convergence of the non-linear term
$$
\lim_{\e\to 0}\int_0^T\int_{\R^2}v^\e\otimes v^\e:\nabla\Phi  \de x \de t =\int_0^T\int_{\R^2}v\otimes v:\nabla\Phi  \de x \de t .
$$
By the special structure of the non-linearity in two dimension, it is sufficient to prove the following convergence, see \cite{De, M}
$$
\lim_{\e\to 0}\int_0^T\int_{\R^2}v^\e_1(t,x)v^\e_2(t,x)\psi(t)\varphi(x)\de x \de t=\int_0^T\int_{\R^2}v_1(t,x)v_2(t,x)\psi(t)\varphi(x)\de x \de t,
$$
for any $\psi\in C^\infty_c(0,T)$ and $\varphi\in C^\infty_c(\R^2)$. We rewrite the left hand side as 
\begin{align*}
\int_0^T\int_{\R^2}v^\e_1(t,x)v^\e_2(t,x)\psi(t)&\varphi(x)\de x \de t \\ &=\int_0^T\int_{\R^2}\int_{\R^2}\psi(t)\omega^\e(t,x)\omega^\e(t,y)H_\varphi(x,y)\de x \de y\de t,
\end{align*}
where 
$$
H_\varphi(x,y)=c\,\, \mathrm{p.v.}\int_{\R^2} \frac{x_1-z_1}{|x-z|^2}\frac{y_2-z_2}{|y-z|^2} \varphi(z)\de z,
$$
for some constant $c>0$. As shown in \cite[Proposition 1.2.3]{De}, the function $H_\varphi \in L^\infty(\R^2\times\R^2)$, is continuous outside the diagonal of $\R^2\times\R^2$ and goes to $0$ at infinity. Moreover, by Lemma \ref{lem:equi} we know that 
$$
\omega^\e\weaktos \omega \hspace{0.5cm}\mbox{in }L^\infty((0,T);L^1(\R^2)),
$$
and then following the proof of Theorem 1 in \cite{VW} it is not difficult to prove that
\begin{align*}
\lim_{\e\to 0}\int_0^T\int_{\R^2}\int_{\R^2}\psi(t)\omega^\e(t,x)&\omega^\e(t,y)H_\varphi(x,y)\de x \de y\de t\\ &=\int_0^T\int_{\R^2}\int_{\R^2}\psi(t)\omega(t,x)\omega(t,y)H_\varphi(x,y)\de x \de y\de t,
\end{align*}
which is enough to conclude.
\end{proof}

\medskip

\section{Convergence to Lagrangian solutions}
In this section we prove that VB-solutions satisfy the 2D Euler equations in the Lagrangian and renormalized sense. Let us start with the following lemma.
\begin{lem}\label{lem:Kcomp}
Let $K$ be the 2D Biot-Savart kernel and denote by $\tau_aK(x)=K(x-a)$. Then for any $1<r<2$ and all $a\in\R^2$
\begin{equation}\label{es:diffk}
\|\tau_aK-K\|_{L^r}\leq C(r)|a|^\alpha
\end{equation}
where $\alpha=2/r-1$. Moreover choosing $p,q$ such that
\begin{equation}\label{qp}
1+\frac{1}{q}-\frac{1}{p}> \frac{1}{2},
\end{equation}
if $\{ u^\e\}\subset L^p(\R^2)$ is uniformly bounded in $\e$, then the sequence $K*u^\e$ is relatively sequentally compact in $L^q_{\mathrm{loc}}(\R^2)$.
\end{lem}
\begin{proof}
We start by proving \eqref{es:diffk}. Fix $a\in \R^2$ with $a\neq 0$. For $|x|>2|a|$, we have for all $0\leq\theta\leq 1$, $|x+\theta a|>|x|-\theta|a|>\frac{|x|}{2}$, thus we have that
$$
|\tau_aK(x)-K(x)|\leq |a|\sup_{0\leq\theta\leq 1}|\nabla K(x+\theta a|\leq |a|\sup_{0\leq\theta\leq 1}\frac{C}{|x+\theta a|^2}\leq C\frac{|a|}{|x|^2}.
$$
Then we estimate
\begin{equation}\label{es:far}
\int_{|x|>2|a|}|\tau_aK(x)-K(x)|^r\de x\leq C\int_{|x|>2|a|}\frac{|a|^r}{|x|^2r}\de x=C(r)|a|^{2-r}.
\end{equation}
Next, for $|x|\leq 2|a|$ we have 
$$
|\tau_aK(x)-K(x)|\leq \frac{1}{|x+a|}+\frac{1}{|x|}
$$
and then
\begin{align}\label{es:near}
\int_{|x|\leq 2|a|}|\tau_aK(x)-K(x)|^r\de x &\leq \int_{|x|\leq 2|a|}\frac{1}{|x+a|^r}+\frac{1}{|x|^r}\de x \nonumber \\ &\leq 2\int_{|x|\leq 3|a|}\frac{1}{|x|^r}\de x =C(r)|a|^{2-r}.
\end{align}
Combining \eqref{es:far} and \eqref{es:near} we get \eqref{es:diffk}.\\
To prove the compactness we want to verify the hypotesis of the Fr\'echet-Kolmogorov theorem. Let $u^\e$ be a bounded sequence in $L^p(\R^2)$. We want to prove that
$$
\lim_{a\to 0}\|\tau_a (K*u^\e)-(K*u^\e)\|_{L^q}=0 \hspace{0.3cm}\mbox{uniformly in }\e.
$$
Thanks to the properties of the convolution, we have that
\begin{align*}
\|\tau_a (K*u^\e)-(K*u^\e)\|_{L^q} &=\|(\tau_a K-K)*u^\e\|_{L^q}\\ &\leq \|\tau_a K-K\|_{L^r}\|u^\e\|_{L^p} \\ &\leq C(r)|a|^{\frac{2}{r}-1}
\end{align*}
which concludes the proof since our choice of $p,q$ implies that $1<r<2$.
\end{proof}
We summarize in the following lemma the convergence of the vortex-blob method to VB-solutions.
\begin{lem}\label{lem:conv}
Let $v$ be a VB-solution and let $\{(\omega^\e,v^\e)\}_\e$ be the approximate vorticity and velocity constructed by the vortex-blob approximation as in the Definition \ref{def:VB}. Let $\omega_0\in L^1_c(\R^2)\cap H^{-1}_{\mathrm{loc}}(\R^2)$ or $\omega_0\in L^p_c(\R^2)$ for $p>1$, then there exists 
$$
\omega\in L^\infty((0,T);L^p(\R^2))
$$
such that up to subsequences the following hold true
\begin{itemize}
\item[(i)] if $p>1$, then $v$ satisfies (R2a) and
$$
v^\e\to v \hspace{1cm} in \ \ L^2((0,T);L^2_{\mathrm{loc}}(\R^2)),
$$
\item[(ii)] if $p=1$, then $v$ satisfies (R2b) and for every $1\leq q<2$
$$
v^\e\to v \hspace{1cm} in \ \ L^q((0,T);L^q_{\mathrm{loc}}(\R^2)),
$$
\item[(iii)] $\omega^\e\weaktos\omega \hspace{1cm} in \ \ L^\infty((0,T);L^p(\R^2))$.
\end{itemize}
\end{lem}
\smallskip
\begin{proof}
We divide the proof in several steps.\\
\\
\underline{Step 1} \hspace{0.5cm} \emph{Convergence of the vorticity.}\\
\\
By Lemma \ref{lem:bound}, we have that the approximate vorticity satisfies
\begin{equation}
\sup_\e \sup_{t\in[0,T]}\int_{\R^2}|\omega(t,x)|^p\de x \leq C.
\end{equation}
Moreover, when $p=1$ we also have by Lemma \ref{lem:equi} that $\{\omega^\e\}$ is equi-integrable. Then, there exists $\omega\in L^\infty((0,T);L^p(\R^2))$ such that 
\begin{equation}\label{wconvvb}
\omega^\e\weaktos \omega \hspace{0.5cm} \mbox{in }  L^\infty((0,T);L^p(\R^2)).
\end{equation}
\medskip
\\
\underline{Step 2} \hspace{0.5cm} \emph{Convergence of the velocity.}\\
\\
The approximate velocity $v^\e$ satisfies the following uniform bound
\begin{equation}\label{estenv}
\sup_\e \sup_{t\in[0,T]}\int_{B_R}|v^\e(t,x)|^2\de x\leq C(R),
\end{equation}
as a consequence of Young's inequality in the case $p>1$ and of Proposition~\ref{prop:vL2loc} for $p=1$. Moreover, since $v$ is a VB-solution, we have that
\begin{equation}\label{convv}
v^\e\weaktos v \hspace{0.5cm} \mbox{in } L^\infty((0,T);L^2_{\mathrm{loc}}(\R^2)).
\end{equation}
In addition, for some $s,r>0$ we also have the following uniform bound 
$$
\{v^\e\}\subset \mathrm{Lip}([0,T];W^{-s,r}(\R^2)),
$$
(see \cite{Be}). Then, thanks to Aubin-Lions' Lemma together with Lemma \ref{lem:Kcomp}, for $p>1$ we can upgrade the convergence \eqref{convv} to
$$
v^\e\to v \hspace{0.5cm} \mbox{in } L^2((0,T);L^2_{\mathrm{loc}}(\R^2)),
$$
while for $p=1$ we have
$$
v^\e\to v \hspace{0.5cm} \mbox{in } L^q((0,T);L^q_{\mathrm{loc}}(\R^2)),
$$
for any $1\leq q<2$, and this concludes the proof.
\end{proof}
We can now prove our first main theorem.
\begin{thm}\label{teo:main1}
Let $v$ be a VB-solution. Then $v$ satisfies the Euler equations in the sense of Lagrangian and renormalized solutions.
\end{thm}
\begin{proof}
We divide the proof in two steps.\\
\\
\underline{Step 1} \hspace{0.5cm} \emph{Representation formula and additional regularity of $v$.}\\
\\
Let $(\omega^\e,v^\e)$ be a sequence constructed via the vortex blob method which converges to $(\omega,v)$ as in Lemma \ref{lem:conv}. We want to prove that $v=K*\omega$ a.e. in $(0,T)\times\R^2$. For $\eta\in C^\infty_c((0,T)\times\R^2)$ we have
\begin{align*}
0&=\lim_{\e\to 0}\iint (v^\e-K*\omega^\e)\eta \de x\de t =\lim_{\e\to 0}\iint v^\e\eta-\omega^\e(K*\eta) \de x \de t \\ &=\iint v\eta-\omega(K*\eta) \de x \de t =\iint (v-K*\omega)\eta \de x\de t,
\end{align*}
where we have used the fact that $K*\eta\in L^q$ for every $1\leq q\leq \infty$. By varing $\eta\in C^\infty_c((0,T)\times\R^2)$ we have the result.
Moreover, the gradient of $v$ can be written as 
$$
\left(\nabla v\right)_{ij}=S^i_j\omega \hspace{1cm}\mbox{in }\mathcal{S}'(\R^2),\hspace{0.5cm}i,j=1,2,
$$
where each $S^i_j$ is a singular integral operator of fundamental type with kernel the distributional derivative $\partial_{x_j} K_i$. Hence $v$ satisfies hypothesis (R2b) if $\omega_0\in L^1$ since $\partial_{x_j} K_i$ define singular integral operators of fundamental type (see Remark 2.11 in \cite{BC} for the definition of a singular integral on $L^1$ functions). In the case $p>1$, by standard Calder\`on-Zygmund theory on singular integrals we have the estimate 
$$
\sup_{t\in[0,T]}\|\nabla v(t,\cdot)\|_{L^p}\leq\sup_{t\in[0,T]}\|\omega(t,\cdot)\|_{L^p}\leq C\|\omega_0\|_{L^p}
$$
and then $v$ satisfies (R2a). \\
\\
\underline{Step 2}\hspace{0.5cm} \emph{Lagrangian property of the solution.}\\
\\
Let $(\omega^\e,v^\e)$ be chosen as in the previous step and consider the auxiliary problem \eqref{eq:1} introduced in Section 3. By Theorem \ref{prop:stab} we have the existence of $\bar{\omega}\in C([0,T];L^p(\R^2))$ such that
$$
\bar{\omega}^\e\to\bar{\omega} \hspace{0.5cm}\mbox{in }C([0,T];L^p(\R^2))
$$
where $\bar{\omega}(t,x)=\omega_0(X^{-1}(t,\cdot)(x))$ and $X$ is the unique regular Lagrangian flow of $v$. In order to conclude we want to prove that $\omega=\bar{\omega}$ a.e.. Let $\chi\in C^\infty_c((0,T)\times\R^2)$ and compute
\begin{align*}
&\int_0^T\int_{\R^2} (\omega-\bar{\omega})\chi\de x \de t =\lim_{\e\to 0}\int_0^T\int_{\R^2} (\omega^\e-\bar{\omega}^\e)\chi\de x \de t \\
&=\lim_{\e\to 0}\int_0^T\int_{\R^2} (\omega^\e-\varphi_\e*\bar{\omega}^\e)\chi\de x \de t + \lim_{\e\to 0}\int_0^T\int_{\R^2} (\varphi_\e*\bar{\omega}^\e-\bar{\omega}^\e)\chi\de x \de t.
\end{align*}
By estimate \eqref{es:1} and standard properties of the convolution, it is easy to check that the previous sum goes to $0$ as $\e\to 0$, and by varying $\chi\in C^\infty_c((0,T)\times\R^2)$ we have that $\omega=\bar{\omega}$ a.e. in $(0,T)\times\R^2$.
\end{proof}
\medskip
\begin{rem}\label{rem:strongo}
Note that the Step 2 of the proof of Theorem \ref{teo:main1}  together with the estimate in Lemma \ref{es:1} give that the convergence in \eqref{wconvvb} of the approximate vorticity $\omega^\e$ towards the Lagrangian solution $\omega$ is actually strong.
\end{rem}

%\begin{rem}
%Note that in the previous proof the equi-integrability of $\omega^\e$ is crucial not only to guarantee that the weak limit is in $L^1$, but it also allow us to use the stability Theorem \ref{prop:stab} since it is actually proved for regularizations of the vector field, which are trivially equi-integrable.
%\end{rem}
\bigskip

\section{Conservation of the energy}
In this section we prove our second main result, namely the conservation of the kinetic energy for VB-solutions. We recall the definition of \emph{conservative weak solution} of the 2D Euler equations from \cite{CFLS}.
\begin{defn}
Let $v\in C([0,T];L^2(\R^2))$ be a weak solution of \eqref{eq:eu} with initial datum $v_0\in L^2(\R^2)$. We call $v$ a \emph{conservative weak solution} if
$$
\|v(t,\cdot)\|_{L^2}=\|v_0\|_{L^2} \hspace{1cm} \forall t\in [0,T].
$$
\end{defn}
First of all, note that in the previous definition we are dealing with initial data which are globally square integrable in space, which is equivalent to requiring that the vorticity has zero mean value. This is the content of the following proposition, which can be found in \cite[Prop. 3.3]{MB}:
\begin{prop}\label{prop:dpm}
An incompressible velocity field in $\R^2$ with vorticity of compact support has finite kinetic energy if and only if the vorticity has zero mean value, that is
\begin{equation}\label{zeromean}
\int_{\R^2}|v(t,x)|^2\de x<\infty\iff\int_{\R^2}\omega(t,x)\de x=0.
\end{equation}
\end{prop}
Before continuing with our result on the vortex-blob approximation, we recall a theorem proved in \cite{CFLS} about the conservation of the energy for weak solutions of the $2D$ Euler equations. This will be useful in order to better understand our result.
\begin{thm}\label{teo:consweak}
Fix $T>0$ and let $v\in C([0,T];L^2(\T))$ be a weak solution of the 2D Euler equations \eqref{eq:eu} with $\omega\in L^\infty((0,T);L^{p}(\T))$ with $p\geq 3/2$. Then $v$ is conservative. Moreover, the following local energy balance law holds in the sense of distributions
$$
\partial_t\left(\frac{|v|^2}{2}\right)+\dive \left(v\left(\frac{|v|^2}{2}+p\right)\right)=0.
$$
\end{thm}
Note that in the previous theorem assumption \eqref{zeromean} is not needed since $\T$ is a bounded domain and then $v_0\in L^2(\T)$ even if the vorticity does not have zero mean value. The method of the proof is based on a mollification argument and the exponent $3/2$ is sharp for the method. In particular, the theorem is still valid if we consider weak solutions $v\in C([0,T];L^2(\R^2))$, with zero mean, such that $\omega\in L^\infty((0,T);L^1\cap L^p(\R^2))$ with $p\geq 3/2$.\\
We now prove that under hypothesis \eqref{zeromean} the approximate velocity given by the vortex-blob method is globally square integrable in space.
\begin{lem}\label{lem:glob2}
Let $\omega_0\in L^1_c(\R^2)$ which verifies \eqref{zeromean}. Then the velocity field $v^\e$ given by \eqref{eq:av} verifies the following uniform bound
$$
\sup_{t\in[0,T]}\|v^\e(t,\cdot)\|_{L^2}\leq C,
$$
provided that $\delta(\e)=\e^\sigma$ with $0<\sigma<1/7$.
\end{lem}
\begin{proof}
Multiply the equation \eqref{eq:vbv} by $v^\e$; integrating over the whole plane and by using the notation $\star$ introduced in Section 2.1, we obtain
\begin{align*}
\frac{1}{2}\frac{\de}{\de s}\|v^\e(s,\cdot)\|_{L^2}^2 &= \int_{\R^2}(K*E_\e)\cdot v^\e\de x=-\int_{\R^2}E_\e(K\star v^\e)\de x\\
& =\int_{\R^2}F_\e\cdot\nabla (K\star v^\e)\de x=-\int_{\R^2}(\nabla K\star F_\e)\cdot v^\e\de x\\
&\leq \|  \nabla K\star F_\e(s,\cdot)\|_{L^2}\|  v^\e(s,\cdot)\|_{L^2}\leq \|  F_\e(s,\cdot)\|_{L^2}\|  v^\e(s,\cdot)\|_{L^2},
\end{align*}
which means that
$$
\frac{\de}{\de s}\|v^\e(s,\cdot)\|_{L^2}\leq \|  F_\e(s,\cdot)\|_{L^2}.
$$
Integrating in time we have that
$$
\|v^\e(t,\cdot)\|_{L^2}\leq \int_0^T\|  F_\e(s,\cdot)\|_{L^2}\de s+\|v^\e(0,\cdot)\|_{L^2}.
$$
Note that $v^\e(0,\cdot)=K*\omega^\e(0,\cdot)$ verifies the hypothesis of Proposition \ref{prop:dpm} and, since the support of $\omega^\e(0,\cdot)$ is uniformly bounded in $\e$, we have that
$$
\|v^\e(0,\cdot)\|_{L^2}\leq C,
$$
where the constant $C$ is indipendent from $\e$. We omit the details of the proof of the previous inequality since it can be done with the same computations of the bound of the $L^2$ norm of $\tilde{v}_0^\e$ in Proposition \ref{prop:vL2loc}. This fact together with Lemma \ref{lem:fe} gives the result.
\end{proof}

With the previous lemma we can prove that the velocity field $v^\e$ converges globally in $L^2$ towards $v$: this will be fundamental in the proof of Theorem \ref{teo:main2}.
\begin{lem}\label{lem:globalconv}
Let $\omega_0\in L^p_c(\R^2)$, with $p>1$, which verifies \eqref{zeromean}. Let $v$ be a VB-solution and $\{v^{\e}\}_{\e}$ as in Definition \ref{def:VB}. Then, up to a subsequence not relabelled the velocity field $v^\e$  satisfies the following convergence
\begin{equation}\label{conv:glob}
v^\e\to v \hspace{1cm}\mbox{in }C([0,T];L^2(\R^2)).
\end{equation}
\end{lem}
\begin{proof}
According to Lemma \ref{lem:conv} and Remark \ref{rem:strongo}, up to a subsequence not relabelled, there exists $\omega\in C([0,T];L^p(\R^2))$ such that 
\begin{equation*}
\omega^\e\to \omega \hspace{0.5cm} \mbox{in }  C([0,T];L^p(\R^2)).
\end{equation*}
Moreover, by Lemma \ref{lem:glob2} both $v$ and $v^{\e}$ are in $L^{\infty}((0,T);L^2(\R^2))$. 
In order to prove the convergence stated in \eqref{conv:glob}, we will prove that $v^\e$ is a Cauchy sequence in $C([0,T];L^2(\R^2))$. Let $\{\e_n\}_{n}$ be any infinitesimal sequence. We denote $v^n,\omega^n$ the velocity field and the vorticity given by the vortex-blob approximation. We divide the proof in several steps.\\
\\
\underline{Step 1}\hspace{0.5cm}\emph{A Serfati identity for the vortex-blob approximation.}\\
\\
In this step we derive a formula for the approximate velocity $v^n$ in the same spirit of the Serfati identity derived in \cite{AKFL, S}.\\
Let $a\in C^\infty_c(\R^2)$ be a smooth function such that $a(x)=1$ if $|x|<1$ and $a(x)=0$ for $|x|>2$. Differentiating in time the Biot-Savart formula we obtain that for $i=1,2$
\begin{align}
\partial_s v^n_i(s,x)&=K_i*(\partial_s \omega^n)(s,x) \nonumber\\
&=(aK_i)*(\partial_s \omega^n)(s,x)+[(1-a)K_i]*(\partial_s \omega^n)(s,x).\label{proof:convglob1}
\end{align}
Now we use the equation \eqref{eq:wvb} for $\omega^n$ obtaining
$$
\partial_s \omega^n=-v^n\cdot\nabla \omega^n+E_n,
$$
and substituting in \eqref{proof:convglob1} we obtain
\begin{equation}
\partial_s v^n_i=(aK_i)*(\partial_s \omega^n)-[(1-a)K_i]*(v^n\cdot\nabla \omega^n)+[(1-a)K_i]*E_n.
\end{equation}
Since $E_n=\dive F_n$ and by the identity
$$
v^n\cdot\nabla \omega^n=\curl(v^n\cdot\nabla v^n)=\curl\dive(v^n\otimes v^n)
$$
we obtain that
\begin{align}
&[(1-a)K_i]*(v^n\cdot\nabla \omega^n)=\left(\nabla\nabla^\perp[(1-a)K_i]\right)\star(v^n\otimes v^n),\label{proof:convglob2}\\
&[(1-a)K_i]*E_n=\left(\nabla[(1-a)K_i]\right)\star F_n, \label{proof:convglob3}
\end{align}
where the notation $\star$ was already introduced.
Substituting the expressions \eqref{proof:convglob2} and \eqref{proof:convglob3} in \eqref{proof:convglob1} and integrating in time we have that $v^n$ satisfies the following formula:
\begin{equation}\label{eq:serfativb}
\begin{split}
v^n_i(t,x)&=v^n_i(0,x)+(aK_i)*\left(\omega^n(t,\cdot)-\omega^n(0,\cdot)\right)(x)\\
&-\int_0^t\left(\nabla\nabla^\perp[(1-a)K_i]\right)\star(v^n(s,\cdot)\otimes v^n(s,\cdot))(x)\de s\\ &+\int_0^t\left(\left(\nabla[(1-a)K_i]\right)\star F_n(s,\cdot)\right)(x)\de s.
\end{split}
\end{equation}
\\
\underline{Step 2}\hspace{0.5cm}\emph{$v^n$ is a Cauchy sequence in $C([0,T];L^2(\R^2))$.}\\
\\
Using formula \eqref{eq:serfativb} we can prove that $v^n$ is a Cauchy sequence. We consider $v^n,v^m$ with $n,m \in\N$. By linearity of the convolution we have that $v^n-v^m$ satisfies the following
\begin{equation}\label{diff}
\begin{split}
v^n_i(t,x)&-v^m_i(t,x)=\underbrace{v^n_i(0,x)-v^m_i(0,x)}_{(I)}\\ &
+\underbrace{(aK_i)*(\omega^n(t,\cdot)-\omega^m(t,\cdot))(x)}_{(II)}+\underbrace{(aK_i)*(\omega^m(0,\cdot)-\omega^n(0,\cdot))(x)}_{(III)}\\ 
&-\int_0^t\underbrace{\left(\nabla\nabla^\perp[(1-a)K_i]\right)\star(v^n(s,\cdot)\otimes v^n(s,\cdot)-v^m(s,\cdot)\otimes v^m(s,\cdot))(x)}_{(IV)}\de s\\ &+\int_0^t\underbrace{\left(\left(\nabla[(1-a)K_i]\right)\star(F_n(s,\cdot)-F_m(s,\cdot)\right)(x)}_{(V)}\de s.
\end{split}
\end{equation}
In order to estimate $\|v^n(t,\cdot)-v^m(t,\cdot)\|_{L^2}$ we estimate separately the $L^2$ norms of the terms on the right hand side of \eqref{diff}. We start by estimating $(I)$: given $\eta>0$, since the initial datum $v^n(0,\cdot)$ converges in $L^2$ to $v_0$, we have that there exists $N_1$ such that 
\begin{equation}\label{I}
\|v^n(0,\cdot)-v^m(0,\cdot)\|_{L^2}<\eta \hspace{1cm}\mbox{for any }n,m>N_1.
\end{equation}
We deal now with $(II),(III)$: if $\omega_0\in L^p_c(\R^2)$ with $1<p<2$, by Young's convolution inequality we have that
$$
\|(aK)*(\omega^n(t,\cdot)-\omega^m(t,\cdot))\|_{L^2}\leq \|aK\|_{L^q}\|\omega^n(t,\cdot)-\omega^m(t,\cdot)\|_{L^p},
$$
where $1<q<2$ is such that $1+1/2=1/p+1/q$, while for $p\geq 2$
$$
\|(aK)*(\omega^n(t,\cdot)-\omega^m(t,\cdot))\|_{L^2}\leq \|aK\|_{L^1}\|\omega^n(t,\cdot)-\omega^m(t,\cdot)\|_{L^2}. 
$$
Notice that $\|aK\|_{L^q}\leq \|K\|_{L^q(B_2)}$ and $K\in L^q_{\mathrm{loc}}(\R^2)$ for any $1\leq q<2$. Moreover, by the strong convergence of $\omega^n$ in $C((0,T); L^{p}(\R^2))$ and the bound $\{\omega^n\}_{n}\subset L^{\infty}((0,T);L^{1}\cap L^{p}(\R^2))$ we conclude that both in the case $1<p<2$ and in the case $p\geq 2$ there exists $N_2$ such that
\begin{equation}\label{II+III}
%\|(aK)*(\omega^n(t,\cdot)-\omega^m(t,\cdot))(x)\|_{L^2}+\|(aK_i)*(\omega^m(0,\cdot)-\omega^n(0,\cdot))(x)\|_{L^2}<C\eta,
\|(aK)*(\omega^n(t,\cdot)-\omega^m(t,\cdot))\|_{L^2}+\|(aK)*(\omega^n(0,\cdot)-\omega^m(0,\cdot))\|_{L^2}<C\eta,
\end{equation}
for any $n,m>N_2$. We deal now with $(IV)$: by Young's convolution inequality we have that
\begin{align}
\|\nabla\nabla^\perp&[(1-a)K]\star(v^n(s,\cdot)\otimes v^n(s,\cdot)-v^m(s,\cdot)\otimes v^m(s,\cdot)\|_{L^2}\nonumber\\ &\leq \|\nabla\nabla^\perp[(1-a)K]\|_{L^2}\underbrace{\|v^n(s,\cdot)\otimes v^n(s,\cdot)-v^m(s,\cdot)\otimes v^m(s,\cdot)\|_{L^1}}_{(IV*)}\label{arr}.
\end{align}
We add and subtract $v^n(s,\cdot)\otimes v^m(s,\cdot)$ in $(IV*)$ and by H\"older inequality we have
\begin{align*}
\|v^n(s,\cdot)\otimes v^n(s,\cdot)&-v^m(s,\cdot)\otimes v^m(s,\cdot)\|_{L^1}\\ &\leq \left(\|v^n(t,\cdot)\|_{L^2}+\|v^m(t,\cdot)\|_{L^2}\right)\|v^n(s,\cdot)-v^m(s,\cdot)\|_{L^2}.
\end{align*}
For the first factor in \eqref{arr} we have that
$$
\nabla\nabla^\perp[(1-a)K_i]=-(\nabla\nabla^\perp a )K_i-\nabla^\perp a\nabla K_i-\nabla a \nabla^\perp K_i+(1-a)\nabla\nabla^\perp K_i,
$$
and it is easy to see that each term on the right hand side has uniformly bounded $L^2$ norm. Then we have that
\begin{equation}\label{IV}
\begin{split}
\int_0^t\|\nabla\nabla^\perp[(1-a)K]\star(v^n(s,\cdot)&\otimes v^n(s,\cdot)-v^m(s,\cdot)\otimes v^m(s,\cdot)\|_{L^2}\de s\\ & \leq C \|v_0\|_{L^2}\int_0^t\|v^n(s,\cdot)-v^m(s,\cdot)\|_{L^2}\de s.
\end{split}
\end{equation}
Finally, we deal with $(V)$: again by Young's inequality we have that
\begin{align*}
\|\left(\nabla[(1-a)K]\right)&\star(F_n(s,\cdot)-F_m(s,\cdot)\|_{L^2}\\
&\leq \|\nabla[(1-a)K]\|_{L^2}\|F_n(s,\cdot)-F_m(s,\cdot)\|_{L^1}.
\end{align*}
Arguing as for $(IV)$, since $\nabla K$ is in $L^2(B_1^c)$, a straightforward computation shows that $\nabla[(1-a)K]$ is bounded in $L^2$. On the other hand, $F_n$ goes to $0$ in $L^\infty((0,T);L^1(\R^2))$ so there exists $N_3$ such that for all $n,m>N_3$ we have that
\begin{equation}\label{V}
\|\left(\nabla[(1-a)K]\right)\star(F_n(s,\cdot)-F_m(s,\cdot)\|_{L^2}\leq C\eta.
\end{equation}
Then, putting together \eqref{I},\eqref{II+III},\eqref{IV} and \eqref{V} we obtain that for all $n,m>N:=\max\{N_1,N_2,N_3\}$
\begin{equation}
\|v^n(t,\cdot)-v^m(t,\cdot)\|_{L^2}\leq C\left(\eta+\int_0^t\|v^n(s,\cdot)-v^m(s,\cdot)\|_{L^2}\de s\right),
\end{equation}
and by Gronwall's lemma
\begin{equation}\label{proof:cauchyfinal}
\|v^n(t,\cdot)-v^m(t,\cdot)\|_{L^2}\leq C(T)\eta.
\end{equation}
Taking the supremum in time in \eqref{proof:cauchyfinal} we have the result.
\end{proof}

We are now in position of proving our second main theorem
\begin{thm}\label{teo:main2}
Let $v$ be a VB-solution and assume that the initial vorticity $\omega_0\in L^p_c(\R^2)$, with $p>1$, satisfies \eqref{zeromean}. Then $v$ is a conservative weak solution. Moreover, if $p\geq 6/5$ the following local energy balance holds
\begin{equation}\label{eq:balance}
\partial_t\left( \frac{|v|^2}{2} \right)+\dive\left( v\left(\frac{|v|^2}{2} +p\right)\right)=0 \hspace{0.7cm}\mbox{in }\mathcal{D}'(\R^2).
\end{equation}
\end{thm}
\begin{proof}
We divide the proof in two steps.\\
\\
\underline{Step 1} \hspace{0.5cm}\emph{Local balance of the energy.}\\
\\
Since the result for $p\geq 3/2$ is a consequence of Theorem \ref{teo:consweak}, we give the proof under the assumption $6/5\leq p\leq 3/2$. Let $v^\e$ constructed by the vortex-blob method as in the definition of VB-solutions. We have that
\begin{equation}
\label{conv:v}
v^\e\to v \hspace{0.7cm}\mbox{in }L^\infty((0,T);L^q(\R^2)),\hspace{0.5cm}\mbox{for every }2\leq q\leq p^*.
\end{equation}
For $q=2$ the convergence \eqref{conv:v} is a consequence of Lemma \ref{lem:globalconv}, while for $2<q\leq p^*$ it follows from Sobolev inequality and the strong convergence of the vorticity. Indeed, by the Calder\`on-Zygmund theorem we have that
$$
\sup_{t\in[0,T]}\|v^\e(t,\cdot)-v(t,\cdot)\|_{L^{p*}}\leq C\sup_{t\in[0,T]}\|\omega^\e(t,\cdot)-\omega(t,\cdot)\|_{L^p}
$$
and by interpolating the spaces $L^2$ and $L^{p^*}$ the convergence in \eqref{conv:v} holds.\\
The pressure $p^\e$ solves the following equation
$$
-\Delta p^\e=\dive\dive(v^\e\otimes v^\e),
$$
and by elliptic regularity we have that $p^\e\in L^\infty((0,T);L^q(\R^2))$, where $1~\leq ~q \leq ~p^*/2$, with uniform bounds. Therefore there exists a scalar function $p~\in ~L^\infty((0,T);L^1\cap~ L^{\frac{p^*}{2}}~(\R^2))$ such that
\begin{equation}\label{conv:p}
p^\e\weaktos p \hspace{1cm}\mbox{in }L^\infty((0,T);L^q(\R^2)), \hspace{0.5cm}\mbox{for all }1<q\leq \frac{p^*}{2}.
\end{equation}
Let $\phi\in C^\infty_c((0,T)\times\R^2)$ be a test function. Multiplying the equation \eqref{eq:vbv} by $v^\e\phi$ and integrating in space and time we get
\begin{align}
\int_0^t\int_{\R^2}\frac{|v^\e|^2}{2}\partial_s\phi\de x \de s & + \int_0^t\int_{\R^2}v^\e\left(\frac{|v^\e|^2}{2}+p^\e\right)\nabla\phi \de x \de s \label{proof:balancevel}\\
&=-\int_0^t\int_{\R^2} (K*E_\e) v^\e \phi\de x \de s.\label{proof:balanceaerror}
\end{align}
We start by considering the error term in \eqref{proof:balanceaerror}: we have that
\begin{align*}
&\int_0^t\int_{\R^2} (K*E_\e)\cdot v^\e \phi\de x \de s=-\int_0^t\int_{\R^2} E_\e(K\star(v^\e\phi)) \de x \de s \\ &=-\int_0^t\int_{\R^2} (\dive F_\e)(K\star(v^\e\phi)) \de x \de s=\int_0^t\int_{\R^2} F_\e\cdot\nabla(K\star(v^\e\phi)) \de x \de s.
\end{align*}
Then, by H\"older inequality and Calder\`on-Zygmund theorem we have that
\begin{align*}
\begin{vmatrix}
\displaystyle\int_0^t\int_{\R^2} (K*E_\e) v^\e \phi\de x \de s
\end{vmatrix}&\leq T \sup_{t\in[0,T]}\left( \|F_\e(t,\cdot)\|_{L^2}\|\nabla K\star(v^\e\phi)(t,\cdot)\|_{L^2} \right)\\
& \leq CT \sup_{t\in[0,T]}\left( \|F_\e(t,\cdot)\|_{L^2}\|v^\e(t,\cdot)\|_{L^2} \right)\\
&\leq C\delta^{-\frac{7}{3}}\e^{\frac{1}{3}}.
\end{align*}
which goes to $0$ as $\e\to 0$ choosing $\delta,h$ in the construction of the approximation as in Lemma \ref{lem:fe}. We consider now \eqref{proof:balancevel}. By the convergence in \eqref{conv:v} we have that
$$
\int_0^t\int_{\R^2}\frac{|v^\e|^2}{2}\partial_s\phi\de x\to \int_0^t\int_{\R^2}\frac{|v|^2}{2}\partial_s\phi\de x, \hspace{0.5cm}\mbox{as }\e\to 0.
$$
We deal now with the second term in \eqref{proof:balancevel}. It is here that the restriction to $p\geq\frac{6}{5}$ comes into play: in this range the Sobolev exponent $p^*\geq 3$. Then, the convergences in \eqref{conv:v} and \eqref{conv:p} imply that
$$
\int_0^t\int_{\R^2} p^\e v^\e\cdot\nabla\phi \de x \de s\to\int_0^t\int_{\R^2} pv\cdot\nabla\phi \de x \de s, \hspace{0.5cm}\mbox{as }\e\to 0,
$$
and
$$
\int_0^t\int_{\R^2}v^\e \frac{|v^\e|^2}{2}\nabla\phi \de x \de s\to \int_0^t\int_{\R^2}v \frac{|v|^2}{2}\nabla\phi \de x \de s, \hspace{0.5cm}\mbox{as }\e\to 0,
$$
and this concludes the proof of \eqref{eq:balance}.\\
\\
\\
\underline{Step 2} \hspace{0.5cm}\emph{Conservation of the kinetic energy.}\\
\\
We prove now that $v$ is a conservative weak solution for any $p>1$. Multiplying \eqref{eq:vbv} by $v^\e$ and integrating in space and time we have that
\begin{equation}\label{eq:energy}
\int_{\R^2}|v^\e|^2(t,x)\de x =\int_{\R^2}|v^\e|^2(0,x)\de x-\int_0^t\int_{\R^2}(\nabla K\star F_\e)\cdot v^\e\de x.
\end{equation}
For the second term on the right hand side, by Lemma \ref{lem:fe} we have that
\begin{align*}
\begin{vmatrix}
\displaystyle\int_0^t\int_{\R^2}(\nabla K\star F_\e)\cdot v^\e\de x
\end{vmatrix}&\leq \|  \nabla K\star F_\e(s,\cdot)\|_{L^2}\|  v^\e(s,\cdot)\|_{L^2}\\ &\leq \|  F_\e(s,\cdot)\|_{L^2}\|  v^\e(s,\cdot)\|_{L^2}\\ &\leq C\delta^{-\frac{7}{3}}\e^{\frac{1}{3}},
\end{align*}
which goes to $0$ as $\e\to 0$. Then, by the convergence \eqref{conv:glob} and letting $\e\to 0$ in \eqref{eq:energy} we have that 
$$
\int_{\R^2}|v|^2(t,x)\de x =\int_{\R^2}|v_0|^2(x)\de x,
$$
which gives the result.
\end{proof}

\medskip

\subsection*{Concluding remarks}
Note that the previous proof the global convergence of the velocity field in \eqref{conv:glob}, which depends on the strong convergence of the vorticity, allows us to prove the conservation of the energy for $p>1$. In fact, the local balance of the energy \eqref{eq:balance} actually implies the conservation of the $L^2$ norm of $v$ for $p\geq 6/5$ by choosing properly the test functions. For example, we can choose the test function to be $\phi_R(x)=\phi\left(\displaystyle\frac{x}{R}\right)$; letting $\e\to 0$ and then $R\to\infty$ we obtain the result. A suitable modification of our argument allow to prove the convergence \eqref{conv:glob} also for solutions constructed as limit of (ES) and (VV) and this extend the result of \cite{CFLS} to the case of the full plane. This suggests that the three methods are somehow equivalent since, under the same hypothesis, they produce weak solutions that share the properties of being Lagrangian and conservative.

\bigskip
\subsection*{Acknowledgments}
This research has been supported by the ERC Starting Grant 676675 FLIRT. Gennaro Ciampa and Stefano Spirito acknowledge the support of INdAM-GNAMPA.

\bigskip

\end{document}